\documentclass{daj}

\usepackage{amsfonts, amsthm, amsmath, amssymb}

\newtheorem{theorem}{Theorem}[section]
\newtheorem{main}{Theorem}
\newtheorem{fact}[theorem]{Fact}
\newtheorem{lemma}[theorem]{Lemma}
\newtheorem{proposition}[theorem]{Proposition}
\newtheorem{corollary}[theorem]{Corollary}
\newtheorem*{question}{Question}

\newcommand{\sub}{\subseteq}
\newcommand{\sm}{\setminus}
\renewcommand{\d}{\delta}
\newcommand{\FF}{\mathbb{F}}
\newcommand{\KK}{\mathbb{K}}

\newcommand{\NN}{\mathbb{N}}

\newcommand{\B}[1]{\mathbf{{#1}}}
\newcommand{\T}{\mathbf{T}}
\newcommand{\V}[1]{\mathbf{{#1}}}
\newcommand{\VV}{\mathbb{V}}

\newcommand{\class}[1]{{\sf {#1}}}
\newcommand{\poly}{{\sf poly}}

\DeclareMathOperator{\J}{\mathbf{J}}

\DeclareMathOperator{\complexity}{C}
\DeclareMathOperator{\ME}{H_\infty}
\DeclareMathOperator*{\Ex}{\mathbb{E}}
\DeclareMathOperator*{\Exp}{\mathbb{E}}
\DeclareMathOperator{\GL}{GL}

\DeclareMathOperator{\bias}{bias}
\DeclareMathOperator{\codim}{codim}
\DeclareMathOperator{\I}{I}
\DeclareMathOperator{\SR}{SR}
\DeclareMathOperator{\AR}{AR}
\DeclareMathOperator{\GR}{GR}

\DeclareMathOperator{\rank}{rank}

\dajAUTHORdetails{%
  title = {Structure vs.\ Randomness for Bilinear Maps}, %% please capitalize all significant words
  author = {Alex Cohen and Guy Moshkovitz},
  plaintextauthor = {Alex Cohen and Guy Moshkovitz},
  plaintexttitle = {Structure vs. Randomness for Bilinear Maps}, %%  title without math or LaTeX
  keywords = {bilinear complexity, tensors, algebraic geometry},
}   %%% END \dajAUTHORdetails

\dajEDITORdetails{%
   year={2022},
   number={12},
   received={19 August 2021},   % received date: example: 7 January 2017
   published={3 October 2022},  % published date
   doi={10.19086/da.38587},       % XXX = number of paper, e.g. da006 for paper#6
}   %%% END \dajEDITORdetails

\begin{document}

\begin{frontmatter}[classification=text]
%% EDITOR: this will force the keywords to appear right after the Abstract.
%%   If the abstract is too long and would force the keywords off the
%%   front page, please comment out % [classification=text] above
%%   This way the keywords will be floated on the bottom of the first page
%%   even though the Abstract spills over to the next page.

%%% AUTHOR: Title goes here.  This line is optional.  You must use it
%%   if title has footnote attached or requires nontrivial typesetting,
%%   e.g., inclusion of linebreaks to force nice layout.
\title{Structure vs.\ Randomness for Bilinear Maps} %% please capitalize all significant words

%%% AUTHOR:
%%% List all authors. If you wish, place grant acknowledgements in \thanks.
%%% In brackets include a short tag for each author.
\author[alex]{Alex Cohen}
\author[guy]{Guy Moshkovitz\thanks{This research was done as part of the 2020 NYC Discrete Math REU, supported by NSF awards DMS-1802059, DMS-1851420, and DMS-1953141.
		An earlier version of this paper appeared in the proceedings of the 53rd ACM Symposium on Theory of Computing (STOC 2021).}}

%%% AUTHOR: Abstract goes here
\begin{abstract}
We prove that the slice rank of a $3$-tensor (a combinatorial notion introduced by Tao in the context of the cap-set problem), the analytic rank (a Fourier-theoretic notion introduced by Gowers and Wolf), and the geometric rank 
(an algebro-geometric notion introduced by Kopparty, Moshkovitz, and Zuiddam)
are all equal up to an absolute constant.
As a corollary, we obtain strong trade-offs on the arithmetic complexity of a biased bilinear map, and on the separation between computing a bilinear map exactly and on average.
Our result settles open questions of Haramaty and Shpilka [STOC~2010], and of Lovett [Discrete Anal.~2019] for $3$-tensors.
\end{abstract}
\end{frontmatter}

%%% AUTHOR: body of paper starts here
\section{Introduction}
% The body of your paper goes here~\cite{cilleruelo}.

%\newpage %% AUTHOR: please comment out this line.  It serves only
%%   to demonstrate both types of header line in daj-template.pdf

Bilinear maps stand at the forefront of many basic questions in combinatorics and theoretical computer science. 
A bilinear map is, intuitively, just a collection of matrices. 
Formally, a bilinear map $f \colon \FF^{n_1} \times \FF^{n_2} \to \FF^m$, where $\FF$ is any field, is a map $f(\B{x},\B{y})=(f_1(\B{x},\B{y}),\ldots,f_m(\B{x},\B{y}))$ whose every component $f_k$ is a bilinear form $f_k(\B{x},\B{y}) = \sum_{i,j} a_{i,j,k}x_iy_j$, or equivalently, $\B{x}^T A_k \B{y}$ for some matrix $A_k \in \FF^{n_1 \times n_2}$.
While linear maps are thoroughly understood thanks to linear algebra, 
bilinear maps are---in more than one way---still very much a mystery.

\subsection{Structure vs.\ randomness}

In this paper we prove a tight relation between the \emph{slice rank} and the \emph{analytic rank} of bilinear maps, or \emph{$3$-tensors}.
Our proof crucially uses the notion of \emph{geometric rank} as an intermediary, enabling the use of tools from algebraic geometry to ultimately prove that these three notions of rank are in fact equivalent up to a constant.

A $3$-tensor (or sometimes simply a tensor) over a field $\FF$ is a three-dimensional matrix $(a_{i,j,k})_{i,j,k}\in\FF^{n_1 \times n_2 \times n_3}$ with entries $a_{i,j,k} \in \FF$. 
Equivalently, a tensor can be thought of as a degree-$3$ polynomial, namely, a trilinear form $T(\B{x},\B{y},\B{z}) = \sum_{i,j,k} a_{i,j,k} x_iy_jz_k$ with coefficients $a_{i,j,k} \in \FF$, where $\B{x}=(x_1,\ldots,x_{n_1})$, $\B{y}=(y_1,\ldots,y_{n_2})$, $\B{z}=(z_1,\ldots,z_{n_3})$.\footnote{A trilinear form means that every monomial has exactly one variable from $\B{x}$, one from $\B{y}$, and one from $\B{z}$, and so is linear separately in each of $\B{x}$,$\B{y}$, and $\B{z}$.}
Note that a tensor is just a symmetric way to think of a bilinear map 
$f \colon \FF^{n_1} \times \FF^{n_2} \to \FF^{n_3}$, where $f=(f_1,\ldots,f_{n_3})$ with $f_k(\B{x},\B{y}) = \sum_{i,j} a_{i,j,k}x_iy_j$; indeed, each $f_k$ corresponds to a slice $(a_{i,j,k})_{i,j}$ of $T$.\footnote{Yet another point of view is that a $3$-tensor is a member of the vector space $V_1 \otimes V_2 \otimes V_3$, where $V_1=\FF^{n_1},V_2\in\FF^{n_2},V_3=\FF^{n_3}$ are finite-dimensional vector spaces over $\FF$.}
As opposed to matrices, which have only one notion of rank, there are multiple notions of rank for $3$-tensors. 
The notions of rank of $3$-tensors we consider are defined as follows:
\begin{itemize}
	\item The slice rank of $T$, denoted $\SR(T)$, is the smallest $r \in \NN$ such that $T$ can be decomposed as $T = \sum_{i=1}^r f_i g_i$ where $f_i$ is an $\FF$-linear form in either the $\B{x}$, $\B{y}$, or $\B{z}$ variables and $g_i$ is an $\FF$-bilinear form in the remaining two sets of variables, for each $i$.
	\item The analytic rank of $T$ over a finite field $\FF$ is given by $\AR(T) = -\log_{|\FF|}\Exp_{\B{x},\B{y},\B{z}} \chi(T(\B{x},\B{y},\B{z}))$,\footnote{We use $\Exp_\B{x}$ to denote averaging, so $\Exp_\B{x \in \FF^{n}}$ stands for $|\FF|^{-n}\sum_{\B{x} \in \FF^{n}}$.} where we fix $\chi$ to be any nontrivial additive character of $\FF$ (e.g., $\chi(x)=\exp(2\pi i x/p)$ when $\FF=\FF_p$ is prime). 
	\item The geometric rank of $T$, viewed as a bilinear map $f$,\footnote{Although permuting $\B{x},\B{y},\B{z}$ gives rise to three distinct bilinear maps corresponding to $T$, the definition of $\GR(T)$ is invariant under them (see Theorem~3.1 in~\cite{KoppartyMoZu20}).} is defined as $\GR(T) = \codim \ker f$, the codimension of the algebraic variety $\ker f = \{(\B{x},\B{y}) \in \overline{\FF}^{n_1} \times \overline{\FF}^{n_2} \mid f(\B{x},\B{y})=\B{0}\}$. 
\end{itemize}

We note that all three notions above generalize matrix rank. Moreover, just like matrix rank, for any $T \in \FF^{n \times n \times n}$ all three quantities lie in the range $[0,n]$. Furthermore, for the $n \times n \times n$ identity tensor $I_n$ we have $\SR(I_n)=\GR(I_n)=n$ and $\AR(I_n)=(1-o_{|\FF|}(1))n$.
The slice rank was defined by Tao~\cite{Tao16} in the context of the solution of the cap-set problem (similar notions have been considered before in other areas of research). 
The analytic rank was introduced by Gowers and Wolf~\cite{GowersWo11} in the context of higher-order Fourier analysis.
Roughly, this notion measures how close to uniform is the distribution of the values of the polynomial corresponding to the tensor.
The geometric rank is an algebro-geometric notion of rank that was recently introduced in~\cite{KoppartyMoZu20}. 
(A related notion was applied by Schmidt~\cite{Schmidt85} in the context of number theory.)
Intuitively, it measures the number of ``independent'' components of the corresponding bilinear map.
We use it as a geometric analogue 
of the bias of the 
(output distribution of the) 
bilinear map.

Understanding the structure of $d$-tensors, or $d$-dimensional matrices, that have low analytic rank is important in many applications of the structure-vs-randomness dichotomy, in additive combinatorics, coding theory and more (see, e.g.,~\cite{BhowmickLo15,GreenTao09}).
A recent breakthrough, obtained independently by Mili\'{c}evi\'{c}~\cite{Milicevic19} and by Janzer~\cite{Janzer19}, showed that the \emph{partition rank} of a $d$-tensor, which is a generalization of slice rank to $d$-tensors, is bounded from above by roughly $\AR(T)^{2^{2^{\poly(d)}}}$. For fixed $d$ this is a polynomial bound, 
which proves a conjecture of Kazhdan and Ziegler~\cite{KazhdanZi20}. 
Lovett~\cite{Lovett19}, as others have, asks whether in fact a linear upper bound holds.
For $3$-tensors, the best known bound until this work was $\SR(T) \le O(\AR(T)^4)$ by Haramaty and Shpilka~\cite{HaramatySh10}.
They write: ``It is an interesting open question to
decide whether we can do only with the $\sum_{j=1}^{O(\log_{|\FF|}1/\d)} \ell_i \cdot q_i$ part'', which refers to a linear upper bound $\SR(T) \le O(\AR(T))$.
Our main result is as follows.
\begin{main}[Main result]\label{main:summary}
	For any $3$-tensor $T$ over a field $\FF$, 
	$$\SR(T) \le 3\GR(T) \le 8.13\AR(T) $$
	where the first inequality holds over any perfect field\footnote{Commonly considered fields are perfect, including: any field of characteristic zero, any algebraically closed field, and any finite field.}, and the second for any finite field $\FF \neq \FF_2$.
\end{main}

We note that the reverse inequalities are easy: $\GR(T) \le \SR(T)$ (see Theorem~4.1 in~\cite{KoppartyMoZu20}) and $\AR(T) \le \SR(T)$ 
(see Lemma~2.2 of~\cite{KazhdanZi18} or~\cite{Lovett19}).
Thus, as mentioned above, an immediate---and perhaps surprising---corollary of Theorem~\ref{main:summary} is that the combinatorial notion $\SR(T)$, the algebro-geometric notion $\GR(T)$, and the analytic notion $\AR(T)$ are all, up to a constant, equivalent notions of rank. 
In particular, if one wants to estimate the slice rank of a $3$-tensor, as Tao did in a solution of the cap-set problem~\cite{Tao16}, then it is necessary and sufficient to instead estimate 
the bias of the tensor.

\subsection{Complexity vs.\ bias}

The importance of bilinear maps in theoretical computer science cannot be overstated.
One example, in the area of algebraic algorithms, is matrix multiplication. Note that the operation of multiplying two matrices $X,Y \in \FF^{m\times m}$ is a bilinear map $\class{MM}_n\colon \FF^{n}\times\FF^{n}\to\FF^{n}$ with $n=m^2$, as every entry of $XY$ is a bilinear form in the entries of $X$ and $Y$. It has been a persistent challenge to upper bound the \emph{arithmetic complexity} of matrix multiplication, that is, the minimum number of $+,-,\cdot,\div$ operations over $\FF$ required to express $\class{MM}_n$ in terms of its variables. Current research puts the complexity of $\class{MM}_n$ below $O(n^{1.2})$ (the state of the art is $O(n^{1.18643})$ due to Alman and Williams~\cite{AlmanWil20}), with the ultimate goal of getting all the way down to $n^{1+o(1)}$.
For another example of the challenge of bilinear maps, this time in the area of circuit complexity, we mention that explicitly finding even a single bilinear map\footnote{Or, equivalently, a single degree-$3$ polynomial $\sum_{k=1}^N f_k(\B{x},\B{y})z_k$.}
$f\colon\FF^{n}\times\FF^{n}\to\FF^{n}$ with provably superlinear arithmetic complexity, say $\Omega(n^{1.001})$, would imply the first such lower bound in circuit complexity. This should be compared with the fact that almost every bilinear map $f\colon\FF^{n}\times\FF^{n}\to\FF^{n}$ has arithmetic complexity $\Theta(n^2)$.
Finally, in the area of identity testing, it was shown by Valiant~\cite{Valiant79} that identity testing of formulas reduces to deciding whether a given bilinear map $f\colon\FF^{n}\times\FF^{n}\to\FF^{m}$ has full \emph{commutative rank}, meaning a linear combination of its components $f_i$ has full rank $n$. 
Whether this can be decided efficiently remains an open question, 
despite being raised by Edmonds~\cite{Edmonds67} in the early days of computer science, and it has close ties with a variety of other topics, from perfect matchings in bipartite graphs to matrix scaling (see~\cite{GargGuOlWi20}).

Given the importance of bilinear maps, we propose studying other foundational questions of theoretical computer science via the lens of bilinear maps. 
Consider Mahaney's Theorem~\cite{Mahaney82}, a classical result in computational complexity.
It states that, assuming $\class{P}\neq~\class{NP}$, 
no $\class{NP}$-hard language is \emph{sparse}.
Phrased differently, if a boolean function $f\colon\{0,1\}^*\to\{0,1\}$ is 
``extremely'' biased in the sense that
$|f^{-1}(1) \cap \{0,1\}^n| \le \poly(n)$, 
then it is not  $\class{NP}$-hard.
Multiple other classical results in the same vein have been proved (see \cite{Fortune79,HartmanisImSe85,OgiwaraWa91,CaiSi95,CaiSi99,Grochow16}), giving implications of 
such extreme bias
for various complexity classes.
This raises the following fundamental question.
\begin{question}
	For a given class of functions, equipped with notions of complexity and bias,
	what is the best complexity upper bound in terms of bias?\footnote{Contrapositively, this quantifies the phenomenon of high complexity functions exhibiting little bias.} 
\end{question}

We make progress towards the above Question for the class of bilinear maps $f$; 
our notion of complexity is multiplicative complexity $\complexity^*(f)$, which is the number of (non-scalar) multiplications needed to compute $f$ by an arithmetic circuit; 
our notion of bias is the \emph{min-entropy} $\ME(f)$ of the output distribution of $f$.
See Section~\ref{sec:apps} for a more formal discussion and the proof.

\begin{proposition}\label{main:complexity-bias}
	For any bilinear map $f \colon \FF^{n} \times \FF^n \to \FF^n$ over any 
	finite field $\FF \neq \FF_2$,
	$$\complexity^*(f) = O\Big(\frac{\ME(f)}{\log_2|\FF|}n\Big) .$$
\end{proposition}

Another closely related classical result says that, assuming $\class{P} \neq \class{NP}$, it is impossible to efficiently solve $\class{SAT}$ on all but at most polynomially many inputs (this is sometimes phrased as saying that $\class{SAT}$ is not ``$\class{P}$-close'').
Again, this raises a fundamental question: 
For a given class of functions, what is the best possible approximation of a hard function?
Put differently, what is the best possible worst-case to average-case reduction?
For bilinear maps, we give an optimal answer to this question;
again, see Section~\ref{sec:apps} for a more formal discussion and the proof.
We say that maps $f,g$ are \emph{$\d$-close} if $\Pr_{x}[f(x) = g(x)] = \d$,\footnote{We use $\Pr_x$ to denote probability under a uniform choice of $x$.}
and denote by $\SR(f)$ the slice rank of the $3$-tensor corresponding to a bilinear map $f$.

\begin{proposition}\label{main:reduction}
	Let $\FF \neq \FF_2$ be a finite field.
	Any two bilinear maps $f,g\colon \FF^n\to\FF^m$ that are $\d$-close satisfy
	$$|\SR(f) - \SR(g)| \le O(\log_{|\FF|}(1/\d)) .$$
	Moreover, this bound is best possible up to the implicit absolute constant.
\end{proposition}

We note that Kaufman and Lovett~\cite{KaufmanLo08} prove such a reduction for degree-$d$ polynomials over general finite fields, improving a previous result by Green and Tao~\cite{GreenTao09}. However, their reduction is qualitative in nature and the implied bounds are far from optimal (see the next subsection for more discussion on previous results and techniques).

\subsection {Proof overview}

Our proof for the bound $\SR(T) = O(\AR(T))$ in Theorem~\ref{main:summary} (and ultimately for complexity-vs-bias trade-offs for bilinear maps in Section~\ref{sec:apps}) 
goes through an algebraically closed field---despite the statement ostensibly being about polynomials over finite fields.
We use the concepts of dimension and tangent spaces from algebraic geometry to obtain our slice rank decomposition, which ends up yielding the bound $\SR(T)=O(\GR(T))$ (Theorem~\ref{theo:main}). To finish the proof of Theorem~\ref{main:summary}, we prove a new generalization of the Schwartz-Zippel lemma appropriate for our setting, which yields $\GR(T) = O(\AR(T))$ (Proposition~\ref{prop:GR-AR}).

To obtain the slice rank decomposition mentioned above, we first prove a result about linear spaces of matrices in which low-rank matrices have high dimension: 
we show that in any such space, one can always find a somewhat large \emph{decomposable} subspace (Proposition~\ref{theo:core}, Item~(\ref{item:codim-bd})); following Atkinson and Lloyd~\cite{AtkinsonLl80} (see also~\cite{FortinRe04}), a space of matrices is decomposable if, roughly speaking, there is a basis where all the matrices have the same block of zeros. 
This result is proved by looking at a tangent space to a determinantal variety at a non-singular point, which turns out to be a decomposable matrix space in the sense above (Proposition~\ref{fact:tangent-matrices}). 
We further show 
that such a matrix space can be thought of as a $3$-tensor of low slice rank (Lemma~\ref{lemma:SR-Mr}). 
Since the above is proved by applying algebro-geometric tools on varieties, which naturally live over the algebraic closure of our finite field, the resulting slice rank decomposition has coefficients in the closure; however, using a result of Derksen~\cite{Derksen20}, one can convert a small slice rank decomposition over the closure into a decomposition over the base field (Proposition~\ref{prop:stable-rk}).
Finally, we use a result of~\cite{KoppartyMoZu20} to combine the slice rank information we obtained above into a bound on the geometric rank (Fact~\ref{fact:GR}).

We note that our arguments diverge from proofs used in previous works. In particular, we do not use results from additive combinatorics at all, nor do we use any ``regularity lemma'' for polynomials or notions of quasi-randomness.  Instead, our arguments use a combination of algebraic and geometric ideas, which perhaps helps explain why we are able to obtain linear upper bounds.

\paragraph{Paper organization.}
We begin Section~\ref{sec:tan} by giving a brief review of a few basic concepts from algebraic geometry, then determine the behavior of certain tangent spaces, and end by proving a slice rank upper bound related to these tangent spaces.
The first and second inequalities of Theorem~\ref{main:summary} are proved in Section~\ref{sec:SR-GR} and in Section~\ref{sec:GR-AR}, respectively.
In Section~\ref{sec:apps} we prove Proposition~\ref{main:complexity-bias} 
and Proposition~\ref{main:reduction}. 
We end with some discussion and open questions in Section~\ref{sec:open}.

\section{Tangent spaces and slice rank}\label{sec:tan}

\subsection{Algebraic geometry essentials}

We will need only a very small number of basic concepts from algebraic geometry, which we quickly review next. All the material here can be found in standard textbooks (e.g.,~\cite{Harris,Shafarevich}).
A \emph{variety} $\V{V}$ is the set of solutions, in an algebraically closed field, of some finite set of polynomials.
More formally, for a field $\FF$,\footnote{We henceforth denote by $\overline{\FF}$ the algebraic closure of the field $\FF$.}
the variety $\V{V} \sub \overline{\FF}^n$ \emph{cut out}
by the polynomials $f_1,\ldots,f_m \in \FF[x_1,\ldots,x_n]$ is 
$$\V{V} = \VV(f_1,\ldots,f_m) := \{ \B{x} \in \overline{\FF}^n \mid  f_1(\B{x})=\cdots=f_m(\B{x})=0\} .$$
We say that $\V{V} \sub \overline{\FF}^n$ is \emph{defined} over $\FF$ as it can be cut out by polynomials whose coefficients lie in $\FF$.
The ideal of $\V{V}$ is $\I(\V{V})=\{f \in \FF[\B{x}] \mid \forall p \in \V{V} \colon f(p)=0\}$.
Any variety $\V{V}$ can be uniquely written as the union of \emph{irreducible} varieties, where a variety is said to be irreducible if it cannot be written as the union of strictly contained varieties.
The \emph{dimension} of a variety $\V{V}$, denoted $\dim\V{V}$, is the maximal length $d$ of a chain of irreducible varieties 
$\emptyset \neq \V{V}_1 \subsetneq\cdots\subsetneq \V{V}_d \subsetneq \V{V}$.
The \emph{codimension} of $\V{V} \sub \overline{\FF}^n$ is simply $\codim\V{V}=n-\dim\V{V}$.

\subsection{Notation}\label{subsec:notation}

In the rest of the paper we will often find it convenient to identify, with a slight abuse of notation,
a bilinear map $f \colon \FF^{n_1} \times \FF^{n_2} \to \FF^m$ (or tensor) with a linear subspace of matrices, or \emph{matrix space}, $\V{L} \preceq \FF^{n_1 \times n_2}$.
If $f=(f_1,\ldots,f_m)$, we will identify $f$ with the linear subspace $\V{L}$ spanned by the $m$ matrices corresponding to the bilinear forms $f_1,\ldots,f_m$.
Note that this identification is not a correspondence, as it involves choosing a basis for $\V{L}$.  
Importantly, however, since the notions of tensor rank that we study are invariant under the action of the general linear group $\mathrm{GL}_n$ on each of the axes, the choice of basis we make is immaterial in the definition of rank, meaning that $\GR(\V{L})$, $\SR(\V{L})$, $\AR(\V{L})$ are nevertheless well defined. 

For the reader's convenience, we summarize below the different perspectives of tensor/bilinear map/matrix space that we use, and how they relate to each other:
\begin{itemize}
	\item A tensor $T = (a_{i,j,k}) \in \FF^{n_1\times n_2 \times n_3}$,  
	or a multilinear form $T(\B{x},\B{y},\B{z}) = \sum_{i,j,k} a_{i,j,k}x_iy_jz_k$.
	\item A bilinear form $f=(f_1,\ldots,f_{n_3}) \colon \FF^{n_1} \times \FF^{n_2} \to \FF^{n_3}$ with $f_k(\B{x},\B{y}) = \sum_{i,j} a_{i,j,k}x_iy_j$.
	\item A matrix space $\V{L} \preceq \FF^{n_1 \times n_2}$ spanned by $\{A_1,\ldots,A_{n_3}\}$ where $A_k = (a_{i,j,k})_{i,j}$.
\end{itemize}

\subsection{Tangent spaces of a variety}

For a variety $\V{V} \sub \KK^n$, the tangent space $\T_p\V{V}$ to $\V{V}$ at the point $p \in \V{V}$ is the linear subspace
$$\T_p\V{V} = \Big\{ \B{v} \in \KK^n \,\Big\vert\, \forall g \in \I(\V{V}) \colon \frac{\partial g}{\partial \B{v}}(p) = 0 \Big\}.$$
Equivalently, for any choice of a generating set $\{g_1,\ldots,g_s\} \sub \KK[x_1,\ldots,x_n]$ for the ideal $\I(\V{V})$ (which is finitely generated by Hilbert's basis theorem), the tangent space at $p \in \V{V}$ is $\T_p\V{V} = \ker \J_p$, where $\J_p$ is the Jacobian matrix
$$\begin{pmatrix}
	\frac{\partial{g_1}}{\partial x_1}(p) & \cdots & \frac{\partial{g_1}}{\partial x_n}(p) \\
	\vdots  & \ddots & \vdots  \\
	\frac{\partial{g_s}}{\partial x_1}(p) & \cdots & \frac{\partial{g_s}}{\partial x_1}(p) 
\end{pmatrix}_{s\times n} .$$

We will need the following basic fact about tangent spaces (for a proof see, e.g., Theorem~2.3 in~\cite{Shafarevich}).
\begin{fact}\label{fact:dim-tangent}
	For any irreducible variety $\V{V}$ and any $p \in \V{V}$ we have
	$\dim \T_p \V{V} \ge \dim \V{V}$.
\end{fact}

We will also need the following easy observation about the interplay between tangents and intersections.

\begin{proposition}\label{prop:tangent-cap}
	For any two varieties $\V{V}$ and $\V{W}$, and any $p \in \V{V} \cap \V{W}$, 
	$$\T_p (\V{V} \cap \V{W}) \sub \T_p\V{V} \,\cap\, \T_p\V{W}.$$
	In particular, if $\V{V} \sub \V{W}$ then $\T_p \V{V} \sub \T_p \V{W}$.
\end{proposition}
\begin{proof}
	We have $\I(\V{V}) \sub \I(\V{V}\cap \V{W})$ and $\I(\V{W}) \sub \I(\V{V}\cap \V{W})$. Therefore, by the definition of a tangent space, for any $p \in \V{V}\cap\V{W}$ we have $\T_p(\V{V}\cap\V{W}) \sub \T_p\V{V}$ and $\T_p(\V{V}\cap\V{W}) \sub \T_p\V{W}$,
	and thus also $\T_p(\V{V}\cap\V{W}) \sub  \T_p\V{V} \cap  \T_p\V{W}$, as claimed.
\end{proof}

\subsection{Slice rank of tangent spaces of determinantal varieties}

We henceforth denote by $\V{M}_r = \V{M}_r(\KK^{m\times n}) \sub \KK^{m \times n}$ the variety of matrices in $\KK^{m\times n}$ of rank at most $r$.
Note that $\V{M}_r$ is indeed a variety, as it is cut out by a finite set of polynomials: all $(r+1)\times(r+1)$ minors.
It is therefore referred to in the literature as a \emph{determinantal variety}.

The following crucial lemma shows that certain tangent spaces of the variety $\V{M}_r=\V{M}_{r}(\KK^{m \times n})$,
which are matrix spaces, 
have a small slice rank (recall Subsection~\ref{subsec:notation} for the terminology).

\begin{lemma}[Slice rank of tangents]\label{lemma:SR-Mr}
	The tangent space to $\V{M}_r=\V{M}_r(\KK^{m \times n})$, for any algebraically closed field $\KK$, at any matrix $A \in \V{M}_r$ with $\rank(A)=r$ satisfies
	$$\SR(\T_A \V{M}_r) \le 2r.$$
\end{lemma}
To prove Lemma~\ref{lemma:SR-Mr} will need the following result,
which explicitly describes the tangent space to $\V{M}_r$ at any matrix of rank exactly $r$. 
It can be deduced from Example~14.16 in~\cite{Harris}.
We prove it below for completeness.
\begin{proposition}[Tangents of determinantal varieties]\label{fact:tangent-matrices}
	The tangent space to $\V{M}_r=\V{M}_r(\KK^{m \times n})$, for any algebraically closed field $\KK$, at any matrix $A \in \V{M}_r$ with $\rank(A)=r$ is
	$$\T_A \V{M}_r = \{C A + A C' \mid C \in \KK^{m\times m}, C' \in \KK^{n \times n}\}.$$
\end{proposition}
\begin{proof}
	It will be convenient to work with the following equivalent definition of a tangent space of a variety $\V{V}$ at a point $p \in \V{V}$;
	$$\T_p\V{V} = \{ \B{v} \in \KK^n \mid \forall g \in \I(\V{V}) \colon g(p+t\B{v})-g(p) \equiv 0 \pmod{t^2}\}.$$
	To see this equivalence, observe that, using the Taylor expansion of the polynomial $g$ at the point $p$, we have 
	$g(p+t\B{v})-g(p) \equiv t\frac{\partial g}{\partial \B{v}}(p) \pmod{t^2}$.
	
	Now, we will use the fact that the $(r+1)\times(r+1)$ minors not only cut out the variety $\V{M}_r=\V{M}_r(\KK^{m\times n})$, but in fact generate the ideal $\I(\V{M}_r)$.
	Indeed, this follows from the fact that the ideal $I$ they generate is prime (\cite{BrunsVe}, Theorem 2.10) and so $\sqrt{I}=I$, together with Hilbert's Nullstellensatz which gives $\I(\V{M}_r)=\sqrt{I}=I$.
	Let $g_{I,J}$ denote the minor of the submatrix whose set of rows and columns are given by $I \sub [m]$ and $J \sub [n]$, respectively. 
	Thus, $\I(\V{M}_r) = \langle g_{I,J} \mid |I|=|J|=r+1 \rangle$.
	
	Since $\rank(A)=r$, there are invertible matrices $P \in \FF^{m\times m}$ and $Q \in \FF^{n \times n}$ such that $A=PI_rQ$, where 
	$$I_r = \begin{pmatrix}
		1 & 0 & \cdots & 0 & 0 & \cdots & \cdots & \cdots & 0 \\
		0 & 1 & \cdots & 0 & 0 & \cdots & \cdots & \cdots & 0 \\
		\vdots & \vdots & \ddots & \vdots & \vdots & \vdots & \vdots & \vdots & \vdots \\
		0 & 0 & \cdots & 1 & 0 & \cdots & \cdots & \cdots & 0 \\
		0 & 0 & \cdots & 0 & 0 & \cdots & \cdots & \cdots & 0\\
		\vdots  & \vdots & \vdots & \vdots & \vdots & \vdots & \vdots & \vdots & \vdots \\
		0 & 0 & \cdots & 0 & 0 & \cdots & \cdots & \cdots & 0
	\end{pmatrix}_{m \times n}$$
	has the $r \times r$ identity matrix as the upper-left submatrix, that is, the submatrix whose set of rows $I$ and set of columns $J$ are $I=[r]$ and $J=[r]$.
	
	Let $X \in \KK^{m \times n}$. Put $Y = P^{-1}XQ^{-1} \in \KK^{m \times n}$. 
	For every $g = g_{I,J}$ with $|I|=|J|=r+1$ 
	we have $g(A)=0$ and 
	$$g(A+tX) = g(P(I_r+tY)Q) = g(P)g(I_r+tY)g(Q).$$
	It follows that
	\begin{center}
		$g(A+tX)-g(A) \equiv 0 \pmod{t^2} \linebreak\quad\text{ if and only if }\quad 
		g(I_r+tY) \equiv 0 \pmod{t^2}$.
	\end{center}
	Write $Y=(y_{i,j})_{i,j}$. 
	Observe that if $I=[r]\cup\{i\}$ and $J=[r]\cup\{j\}$ for some $i>r$ and $j>r$ then $g_{I,J}(I_r+tY) \equiv ty_{i,j} \pmod{t^2}$, and otherwise $g_{I,J}(I_r+tY) \equiv 0 \pmod{t^2}$.
	Thus, $Y$ satisfies $g(I_r+tY) \equiv 0 \pmod{t^2}$ for every $g \in \I(\V{M}_r)$ if and only if $y_{i,j}=0$ for every $i>r$ and $j>r$, or equivalently, $Y=Y_1I_r + I_rY_2$ for some $Y_1 \in \KK^{m \times m}$ and $Y_2 \in \KK^{n \times n}$.
	We deduce 
	\begin{align*}
		\T_A\V{M}_r &= \{ X \in \KK^{m \times n} \mid \forall g \in \I(\V{M}_r) \colon\\
		& g(A+tX)-g(A) \equiv 0 \pmod{t^2}\}\\
		&= \{ PYQ \mid \exists Y_1 \in \KK^{m\times m}, Y_2 \in \KK^{n\times n} \colon Y = Y_1I_r + I_r Y_2\}\\
		&= \{ (PY_1P^{-1})A + A(Q^{-1} Y_2 Q) \mid Y_1 \in \KK^{m\times m}, Y_2 \in \KK^{n\times n}\}\\
		&= \{C A + A C' \mid C \in \KK^{m\times m}, C' \in \KK^{n \times n}\},
	\end{align*}
	completing the proof.
\end{proof}

We note that any matrix $A$ with $\rank(A)=r$ is a nonsingular point of $\V{M}_r$ (i.e., $\dim\T_A \V{M}_r = \dim\V{M}_r$), whereas any matrix $B$ with $\rank(B)<r$ is a singular point, and in fact, $\T_B\V{M}_r(\KK^{m \times n}) = \KK^{m \times n}$. 

\begin{proof}[Proof of Lemma~\ref{lemma:SR-Mr}]
	We identify $A=(a_{i,j}) \in \V{M}_r(\KK^{m \times n})$ with the bilinear form given by $A(\B{x},\B{y}) = \B{x}^T A \B{y} = \sum_{i,j} a_{i,j} x_i y_j$.
	Since $A \in \KK^{m \times n}$ and $\rank(A) \le r$, there are linear forms $f_1(\B{x}),\ldots,f_r(\B{x}) \in \KK[\B{x}]$ and linear forms $g_1(\B{y}),\ldots,g_r(\B{y}) \in \KK[\B{y}]$ such that
	$$A(\B{x},\B{y}) = \sum_{i=1}^r f_i(\B{x})g_i(\B{y}) .$$
	It follows that any matrix of the form $C A + A C'$, with $C \in \KK^{m\times m}$ and $C' \in \KK^{n \times n}$, has a corresponding bilinear form
	\begin{align}\label{eq:decomp}
		\begin{split}
			\B{x}^T (C A + A C') \B{y} 
			&= (C^T\B{x})^T A \B{y} + \B{x}^T A (C'\B{y})\\
			&= \sum_{i=1}^r f_i(C^T\B{x})g_i(\B{y}) + \sum_{i=1}^r f_i(\B{x})g_i(C'\B{y}) .
		\end{split}
	\end{align}
	Now, let $B_1,\ldots,B_d$ be any basis of $\T_A \V{M}_r$. 
	Then $\T_A \V{M}_r$ corresponds to the trilinear form $T=\sum_{k=1}^d z_k B_k(\B{x},\B{y})$ in the variables $\B{x},\B{y},\B{z}$. 
	By Proposition~\ref{fact:tangent-matrices}, 
	for each $k \in [d]$ we can write $B_k = C_kA+AC_k'$ for some $C_k \in \KK^{m\times m}$ and $C'_k \in \KK^{n \times n}$.
	Using the decomposition in~(\ref{eq:decomp}), we obtain the trilinear decomposition
	\begin{align*}
		T &= \sum_{k=1}^d z_k \cdot B_k(\B{x},\B{y})\\
		&= \sum_{k=1}^d z_k\Big(\sum_{i=1}^r f_i(C_k^T\B{x})g_i(\B{y}) + \sum_{i=1}^r f_i(\B{x})g_i(C_k'\B{y})\Big)\\
		&= \sum_{i=1}^r h_i(\B{x},\B{z}) g_i(\B{y}) + \sum_{i=1}^r f_i(\B{x})h'_i(\B{y},\B{z}) 
	\end{align*}
	where
	$$h_i(\B{x},\B{z}) := \sum_{k=1}^d z_kf_i(C_k^T\B{x}),\quad 
	h'_i(\B{y},\B{z}) := \sum_{k=1}^d z_k g_i(C_k'\B{y}) .$$
	Note that each $h_i \in \KK[\B{x},\B{z}]$ and $h'_i \in \KK[\B{y},\B{z}]$ are bilinear forms over $\KK$, 
	and recall that each $f_i \in \KK[\B{x}]$ and $g_i \in \KK[\B{y}]$ are linear forms over $\KK$.
	We deduce that each of the $2r$ summands in the decomposition of $T$ above is a trilinear form of slice rank at most $1$ over $\KK$.
	This completes the proof.
\end{proof}

\section{Slice rank vs. geometric rank}\label{sec:SR-GR}

In this section we prove the core of our main result, linearly bounding the slice rank of a tensor from above by its geometric rank.

\begin{theorem}\label{theo:main}
	For any $3$-tensor $T$ over any perfect field $\FF$,
	$$\SR(T) \le 3\GR(T).$$
\end{theorem}

We in fact get the slightly better constant $2$ instead of $3$ in Theorem~\ref{theo:main}, at the price of allowing the slice rank decomposition to use coefficients from an algebraic extension.

Let $\overline{\SR}(T)$ denote, for a tensor $T$, the slice rank over the algebraic closure of the field of coefficients of $T$. In other words, if $T$ is a tensor over $\FF$ then $\overline{\SR}(T)$ allows coefficients from the algebraic closure $\overline{\FF}$, rather than just from $\FF$, in the decomposition of $T$ into slice-rank one summands.
Clearly, $\overline{\SR}(T) \le \SR(T)$.
We note that for matrices, $\rank$ and $\overline{\rank}$ are equal. 
For tensors we have the following inequality, essentially due to Derksen~\cite{Derksen20} (we include a proof sketch at the end of this section).

\begin{proposition}[\cite{Derksen20}]\label{prop:stable-rk}
For any $3$-tensor $T$ over any perfect field,
$\frac23\SR(T) \le \overline{\SR}(T)$.
\end{proposition}

We will also need the following properties of slice rank, which are easily deduced from definition. 
For convenience of application, we state them for matrix spaces.

\begin{proposition}\label{prop:SR-properties}
The slice rank satisfies the following properties, where $\V{L}$ and $\V{L'}$ are linear subspaces of matrices:
\begin{enumerate}
	\item(Dimension bound) $\SR(\V{L}) \le \dim \V{L}$,
	\item(Monotonicity) $\SR(\V{L'}) \le \SR(\V{L})$ if $\V{L'} \preceq \V{L}$, 
	\item(Sub-additivity) $\SR(\V{L}+\V{L'}) \le \SR(\V{L})+\SR(\V{L'})$.
\end{enumerate}
\end{proposition}

\subsection{Linear sections of determinantal varieties}

For $\V{L} \preceq \KK^{m \times n}$ a matrix space we define the variety $\V{L}_r = \V{L} \cap \V{M}_{r}$ (here $\V{M}_{r}=\V{M}_{r}(\KK^{m \times n})$)
of all matrices in $\V{L}$ of rank at most $r$.
We next bound the slice rank of a matrix space using these linear sections of a determinantal variety. 
We denote by $\codim_L \V{X}$ the codimension of a variety $\V{X} \sub L$ inside a linear space $L$; that is, $\codim_L \V{X} = \dim L - \dim \V{X}$.

\begin{proposition}\label{theo:core}
Let $\V{L} \preceq \KK^{m \times n}$ be a matrix space over any algebraically closed field $\KK$. 
For any $r \in \NN$, 
$$\SR(\V{L}) \le 2r + \codim_\V{L} \V{L}_r.$$
\end{proposition}
\begin{proof}
We proceed by induction on $r$. Note that the base case $r=0$, which reads $\SR(\V{L}) \le 0+\codim_\V{L} \{\B{0}\} = \dim \V{L}$, follows from Proposition~\ref{prop:SR-properties}.
We thus move to the inductive step. 

Let $\V{V}$ be an irreducible component of $\V{L}_r$ with $\dim \V{V} = \dim \V{L}_r$,
and let $A \in \V{V} \sm \V{M}_{r-1}$.
We may indeed assume $\V{V} \sm \V{M}_{r-1} \neq \emptyset$, as otherwise $\V{V} \sub \V{L}_{r-1}$ and thus 
$\dim\V{L}_r = \dim\V{V} \le \dim \V{L}_{r-1}$ 
and we are done via the induction hypothesis by taking codimensions.
Let $\V{P} \preceq \V{L}$ be the linear subspace 
$\V{P} = \V{L} \,\cap\, \T_A\V{M}_r$. 
We will prove:
\begin{enumerate}
	\item\label{item:SR-bd} $\SR(\V{P}) \le 2r$, 
	\item\label{item:codim-bd} $\codim_\V{L} \V{P} \le \codim_\V{L} \V{L}_r$.
\end{enumerate}
To see why this would complete the inductive step, 
let $\V{P}^\perp$ be a complement subspace of $\V{P}$ in $\V{L}$, 
and note that
$$\SR(\V{L}) \le \SR(\V{P}) + \SR(\V{P}^\perp) \le \SR(\V{P}) + \codim_{\V{L}} \V{P} \le 2r + \codim_{\V{L}} \V{L}_r$$
where the first and second inequalities use Proposition~\ref{prop:SR-properties}, and the third inequality uses Items~(\ref{item:SR-bd}) and~(\ref{item:codim-bd}). 

For the proof of Item~(\ref{item:SR-bd}), we have
$$\SR(\V{P}) = \SR(\V{L} \cap \T_A \V{M}_r) \le \SR(\T_A \V{M}_r) \le 2r$$
where the first inequality uses Proposition~\ref{prop:SR-properties}, and the second inequality uses Lemma~\ref{lemma:SR-Mr} as $\rank(A)=r$.
For the proof of Item~(\ref{item:codim-bd}), we have
\begin{align*}
	\dim\V{L}_r = \dim \V{V} &\le \dim \T_A \V{V} \\ 
	&\le \dim \T_A \V{L}_r\\
	&\le \dim( \T_A \V{L} \,\cap\, \T_A \V{M}_r) \\
	&= \dim(\V{L} \,\cap\, \T_A \V{M}_r) 
	= \dim \V{P}
\end{align*}
where the first inequality uses Fact~\ref{fact:dim-tangent}, 
the second inequality uses Proposition~\ref{prop:tangent-cap} together with the fact that $\V{V} \sub \V{L}_r$, 
the third inequality uses Proposition~\ref{prop:tangent-cap} again,
and the last equality uses $\T_A \V{L}=\V{L}$ since $\V{L}$ is a linear subspace.
As the above varieties are subvarieties of $\V{L}$, we obtain
$\codim_{\V{L}} \V{P} \le \codim_{\V{L}} \V{L}_r$.
This proves Item~(\ref{item:codim-bd}) and therefore completes the proof of the inductive step.
\end{proof}

We note that an immediate corollary of Proposition~\ref{theo:core}
is a slice rank upper bound of $2r$ for any subspace of matrices of rank at most $r$.

\subsection{Putting everything together}

To prove Theorem~\ref{theo:main} we also need the following characterization of geometric rank.
Recall that $\GR(T)=\codim \ker T$ where $\ker T=\{(\B{x},\B{y}) \mid T(\B{x},\B{y},\cdot)=\B{0} \}$.
\begin{fact}[\cite{KoppartyMoZu20}]\label{fact:GR}
For any $3$-tensor $T$ over any field,
$$\GR(T) = \min_r \, r + \codim\{\B{x} \mid \rank T(\B{x},\cdot,\cdot)=r\} .$$
\end{fact}
Fact~\ref{fact:GR} is proved via the decomposition
$$\ker T = \bigcup_r \{(\B{x},\B{y}) \in \ker T \mid \rank T(\B{x},\cdot,\cdot)=r\} ,$$
using a result from algebraic geometry on the dimensions of fibers, and the fact that the codimension of a finite union of varieties is the minimum of their codimensions.
We refer to Theorem~3.1 in~\cite{KoppartyMoZu20} for the formal proof.

We are now ready to prove the main result of this section.
First, we show how to obtain Proposition~\ref{prop:stable-rk} from the results in~\cite{Derksen20}.

\begin{proof}[Proof sketch of Proposition~\ref{prop:stable-rk}]
This is obtained by combining Theorem~2.5, Corollary~3.7, and Proposition~4.9 in~\cite{Derksen20}.
These results show that the ``$G$-stable rank'' $\rank^G_\FF(T)$ over a perfect field $\FF$ satisfies the following properties, respectively: 
\begin{itemize}
	\item $\rank^G_\FF(T)=\rank^G_{\overline{\FF}}(T)$,
	\item $\rank^G_\FF(T) \le \SR(T)$,
	\item $\rank^G_\FF(T) \ge (2/3)\SR(T)$.
\end{itemize}
Putting these together gives
$\frac23 \SR(T) \le \rank^G_\FF(T) = \rank^G_{\overline{\FF}}(T) \le \overline{\SR}(T)$, as claimed.
\end{proof}

\begin{proof}[Proof of Theorem~\ref{theo:main}]
Suppose $T=(a_{i,j,k})_{i,j,k} \in \FF^{n_1 \times n_2 \times n_3}$ with $\FF$ an arbitrary field.
Let $\V{L} \preceq \overline{\FF}^{n_2\times n_3}$ be the matrix space spanned by the $n_1$ slices $A_1=(a_{1,j,k})_{j,k},\ldots,A_{n_1}=(a_{n_1,j,k})_{j,k}$. 
Note that we may assume, by acting with general linear group $\GL_{n_1}(\FF)$ on $T$, that the first $d:=\dim\V{L}$ slices $A_1,\ldots,A_{d}$ of $T$ are linearly independent and the rest are zero matrices; indeed, this action does not change $\GR(T)$ (see Lemma~4.2 in~\cite{KoppartyMoZu20}) nor does it change $\SR(T)$.

Note that for any $\B{x} \in \overline{\FF}^{n_1}$, the bilinear form $T(\B{x},\cdot,\cdot)$ corresponds to the matrix $\sum_i x_iA_i$;
indeed, 
\begin{align*}
	T(\B{x},\cdot,\cdot) &\colon (\B{y},\B{z}) \mapsto \sum_{i,j,k} a_{i,j,k}x_iy_jz_k = \sum_i x_i \sum_{j,k} a_{i,j,k}y_jz_k \\
	&= \sum_i x_i \B{y}^TA_i\B{z} = \B{y}^T\Big(\sum_i x_iA_i\Big)\B{z}.
\end{align*}
Using our assumption that $A_i = \B{0}$ for every $i>d$, 
let 
\begin{align*}
	\V{X}_r &= \{\B{x} \in \overline{\FF}^{n_1} \mid \rank T(\B{x},\cdot,\cdot) \le r\}\\ 
	&= \{\B{x} \in \overline{\FF}^{n_1} \mid \rank\big(x_1A_1+\cdots+x_dA_d\big) \le r\}.
\end{align*}
We claim that $\codim\V{X}_r=\codim_\V{L} \V{L}_r$. 
Recall that $\V{L}_r = \{ A \in \V{L} \mid \rank A \le r\}$. 
First, we show that the variety $\V{X}_r$ is isomorphic to the variety $\V{L}_r \times \overline{\FF}^{n_1-d}$. 
Indeed, the polynomial map (in fact linear) 
$$(x_1,\ldots,x_{n_1}) \mapsto (x_1A_1+\cdots+x_dA_d,x_{d+1},\ldots,x_{n_1})$$ 
maps $\V{X}_r$ to $\V{L}_r\times \overline{\FF}^{n_1-d}$, and is invertible via a polynomial map (in fact linear) by our assumption that $A_1,\ldots,A_{d}$ are linearly independent.
We deduce from this isomorphism the equality of dimensions $\dim\V{X}_r = \dim(\V{L}_r\times \overline{\FF}^{n_1-d})$, 
or equivalently, $\codim\V{X}_r = n_1-\dim\V{X}_r = d-\dim\V{L}_r = \codim_\V{L} \V{L}_r$, as claimed.

Let $r$ achieve the minimum in Fact~\ref{fact:GR}. This implies that
$\GR(T) = r + \codim \V{X}_r$. 
By Theorem~\ref{theo:core},
\begin{align*}
	\overline{\SR}(T) &= \SR(\V{L}) \le 2r + \codim_\V{L} \V{L}_r
	= 2r + \codim \V{X}_r\\
	&= 2\GR(T) - \codim \V{X}_r 
	\le 2\GR(T).
\end{align*}
Assuming further that $\FF$ is a perfect field and using Proposition~\ref{prop:stable-rk}, 
we finally obtain the bound
$\SR(T) \le \frac32\overline{\SR}(T) \le 3\GR(T)$, 
as desired.
\end{proof}

\section{Geometric rank vs.\ analytic rank}\label{sec:GR-AR}

Our main result in this section gives an essentially tight upper bound on the geometric rank in terms of the analytic rank. 

\begin{proposition}\label{prop:GR-AR}
For any $3$-tensor $T$ over any finite field $\FF$,
$$\AR(T) \ge (1-\log_{|\FF|}2)\GR(T) .$$
\end{proposition}

\subsection{Schwartz-Zippel meet B\'{e}zout}

We will need a certain generalized version of the classical Schwartz-Zippel lemma that applies to varieties.
We note that there are various generalized versions of the Schwartz-Zippel lemma appearing in the literature (e.g., Lemma~14 in~\cite{BukhTs12}, Claim~7.2 in \cite{DvirKoLo14}, Lemma A.3 in~\cite{EllenbergObTa10}).
However, in our version below the bound goes down exponentially with the codimension of the variety as soon as the field is larger than the degrees of the polynomials cutting out the variety, which is crucial for proving Proposition~\ref{prop:GR-AR}.

We use the notation $\V{V}(\FF) := \V{V} \cap \FF^n$ for any variety $\V{V} \sub \overline{\FF}^n$ defined over $\FF$.
Recall that a variety $\V{V} = \VV(f_1,\ldots,f_m)$ is said to be cut out by the polynomials $f_1,\ldots,f_m$.

\begin{lemma}[Schwartz-Zippel for varieties]\label{lemma:SZ}
Let $\FF$ be a finite field.
For any variety $\V{V} \sub \overline{\FF}^n$
cut out by polynomials of degrees at most $d$, 
$$\frac{|\V{V}(\FF)|}{|\FF|^n} \le \Big(\frac{d}{|\FF|}\Big)^{\codim\V{V}}.$$
\end{lemma}

We note that the classical Schwartz-Zippel lemma is recovered as the special case of Lemma~\ref{lemma:SZ} where $\V{V}$ is cut out by a single polynomial $p$. Indeed, in this case, Lemma~\ref{lemma:SZ} says that if $p$ is a non-zero polynomial, meaning $\codim\V{V}=1$, then $|\V{V}(\FF)|/|\FF|^n \le d/|\FF|$.

Let $\V{V}^0$ denote the union of the $0$-dimensional irreducible components of a variety $\V{V}$. Note that $\V{V}^0$ is a finite set.
For the proof of Lemma~\ref{lemma:SZ} we will use the overdetermined case of B\'{e}zout's inequality, which provides an upper bound on $|\V{V}^0|$ (see~\cite{Tao11}, Theorem 5).

\begin{fact}[Overdetermined B\'{e}zout's inequality]\label{fact:bezout-over}
Let $\V{V} = \VV(f_1,\ldots,f_m) \sub \KK^n$ be a variety, for an algebraically closed field $\KK$, 
cut out by $m \ge n$ polynomials. 
Write $\deg f_1 \ge \cdots \ge \deg f_m \ge 1$. 
Then 
$$|\V{V}^0| \le \prod_{i=1}^n \deg f_i.$$
\end{fact}

The \emph{degree} of an equidimensional\footnote{All irreducible components have the same dimension.} variety $\V{V} \sub \KK^n$, denoted $\deg\V{V}$, is the cardinality of the intersection of $\V{V}$ with a generic linear subspace in $\KK^n$ of dimension $\codim\V{V}$ (a well-defined, finite number).
The degree of an arbitrary variety $\V{V}$ is the sum of the degrees of its irreducible components.
The proof of Lemma~\ref{lemma:SZ} will ``bootstrap'' the following generalization of the Schwartz-Zippel lemma. 

\begin{fact}[\cite{BukhTs12},\cite{DvirKoLo14}]\label{fact:SZ}
Let $\FF$ be a finite field.
For any variety $\V{V}$ over $\overline{\FF}$,
$$|\V{V}(\FF)| \le \deg\V{V}\cdot|\FF|^{\dim\V{V}}.$$
\end{fact}

\begin{proof}[Proof of Lemma~\ref{lemma:SZ}]
We claim that the following inequality holds
assuming $\V{V}$ is equidimensional;
$$\deg \V{V} \le d^{\codim \V{V}}.$$
Suppose $\V{V}$ is cut out by $m$ polynomials of degree at most $d$. 
Note that $m \ge \codim\V{V}$.
Consider the variety obtained by intersecting $\V{V}$ with a generic linear subspace in $\overline{\FF}^n$ of dimension $\codim\V{V}$, 
and observe that it can be embedded as a variety $\V{W} \sub \overline{\FF}^{n_0}$ with $n_0=\codim\V{V}$.
Then $\V{W}$ satisfies the following properties:
\begin{itemize}
	\item $\dim\V{W} = 0$,
	\item $\V{W}$ is cut out by $m$ polynomials of degree at most $d$.
\end{itemize}	
In particular, and similarly to the above, $m \ge n_0$.
It follows that 
$$\deg \V{V} = |\V{W}| = |\V{W}^0| \le d^{n_0} = d^{\codim\V{V}},$$
where the first equality is by the definition of $\deg\V{V}$, the second equality uses $\V{W}=\V{W}^0$ as $\dim\V{W}=0$, 
and the inequality applies Fact~\ref{fact:bezout-over} since we are in the overdetermined case $m \ge n_0$.

To finish the proof, we need to remove the assumption in the above inequality that $\V{V}$ is equidimensional.
This is immediate using the standard technique of replacing the polynomials cutting out $\V{V}$ with generic linear combinations thereof, giving a variety $\V{V'} \supseteq \V{V}$ over $\overline{\FF}$ with $\dim\V{V'}=\dim\V{V}$ that is an equidimensional set-theoretic complete intersection (see, e.g., Kollar~\cite{Kollar88}, Theorem~1.5).\footnote{Alternatively, this can be done by defining an appropriate variant of degree for varieties that are not equidimensional.}
Indeed, 
starting with a single polynomial (which trivially cuts out an equidimensional variety), 
adding each such linear combination $p$, 
that is, intersecting with the hypersurface $\V{H}_p$ corresponding to $p$, reduces by one the dimension of every irreducible component $\V{X}$
(since $\V{X} \nsubseteq \V{H}_p$ as long as $\V{X} \nsubseteq \V{V}$).
Now, apply Fact~\ref{fact:SZ} to obtain
$$\frac{|\V{V}(\FF)|}{|\FF|^n} 
\le \frac{|\V{V'}(\FF)|}{|\FF|^n} 
\le \frac{d^{\codim\V{V}}|\FF|^{\dim\V{V}}}{|\FF|^n}
= \Big(\frac{d}{|\FF|}\Big)^{\codim\V{V}},$$
as desired.
\end{proof}

\subsection{Putting everything together}

We now deduce the desired bound relating $\GR$ and $\AR$.
We will use the following well known characterization of $\AR$; we include a proof for completeness.
\begin{fact}\label{fact:AR-ch}
For any $3$-tensor $T \in \FF^{n_1 \times n_2 \times n_3}$ over any finite field $\FF$,
$$\AR(T) = -\log_{|\FF|} \Pr_{\B{x},\B{y}}[f(\B{x},\B{y})=\B{0}]$$
where $f \colon \FF^{n_1} \times \FF^{n_2} \to \FF^{n_3}$ is the bilinear map corresponding to $T$.
\end{fact}
\begin{proof}
Write $f=(f_1,\ldots,f_{n_3})$, so that $T(\B{x},\B{y},\B{z})=\sum_{k=1}^{n_3} f_k(\B{x},\B{y})z_k$.
Denote $\bias(T) = \Exp_{\B{x},\B{y},\B{z}} \chi(T(\B{x},\B{y},\B{z}))$,
where $\chi$ is an arbitrary, nontrivial additive character of $\FF$,
so that $\AR(T) = -\log_{|\FF|}\bias(T)$.
Since $f$ is bilinear, we have $\bias(T) = 
\Pr_{\B{x},\B{y}} [f(\B{x},\B{y})=\B{0}]$; 
indeed, 
\begin{align*}
	\bias(T) &= \Exp_{\B{x},\B{y},\B{z}} \chi(T(\B{x},\B{y},\B{z}))
	= \Exp_{\B{x},\B{y},\B{z}} \chi\Big(\sum_k f_k(\B{x},\B{y})z_k \Big)\\
	&= \Exp_{\B{x},\B{y}} \Exp_{\B{z}}\prod_k\chi(f_k(\B{x},\B{y}) z_k)
	= \Exp_{\B{x},\B{y}} \prod_k \Exp_{z \in \FF}\chi(f_k(\B{x},\B{y}) z)\\
	&= \Exp_{\B{x},\B{y}} \prod_k [f_k(\B{x},\B{y})=0]
	= \Exp_{\B{x},\B{y}} [f(\B{x},\B{y})=\B{0}]\\
	&= \Pr_{\B{x},\B{y}} [f(\B{x},\B{y})=\B{0}] ,
\end{align*}
where $[\cdot]$ is the Iverson bracket. 
This completes the proof.
\end{proof}

\begin{proof}[Proof of Proposition~\ref{prop:GR-AR}]
Suppose $T \in \FF^{n_1 \times n_2 \times n_3}$. 
Put $\V{V} = \ker(T) \sub \overline{\FF}^{N}$ with $N=n_1+n_2$.
By Lemma~\ref{lemma:SZ},
$|\V{V}(\FF)|/|\FF|^{N} \le (2/|\FF|)^{\codim\V{V}}$.
Using Fact~\ref{fact:AR-ch}, it follows that
$$\AR(T) = -\log_{|\FF|} \frac{|\V{V}(\FF)|}{|\FF|^{N}}
\ge \codim\V{V} \cdot (1-\log_{|\FF|}2) .$$
As $\GR(T) = \codim\V{V}$, we are done.
\end{proof}

We are finally ready to combine our various bounds and obtain the main result.

\begin{proof}[Proof of Theorem~\ref{main:summary}]
The first inequality is given by Theorem~\ref{theo:main}.
The second inequality follows from Proposition~\ref{prop:GR-AR} for any finite $\FF \neq \FF_2$, since
\begin{align*}
	\GR(T) &\le (1-\log_{|\FF|}2)^{-1}\AR(T)\\ 
	&\le (1-\log_{3}2)^{-1}\AR(T) \le 2.71\AR(T).
\end{align*}
\end{proof}
We note that, as evident from the proof of Theorem~\ref{main:summary}, we in fact obtain the bounds $\SR(T) \le 3\GR(T) \le 3(1+o_{|\FF|}(1))\AR(T)$.

\section{Some complexity results for bilinear maps}\label{sec:apps}

\subsection{Rank vs.\ min-entropy}\label{subsec:ME}

Recall that the min-entropy of a discrete random variable $X$ is 
$$\ME(f) = \min_{x} \, \log_2 \frac{1}{\Pr[X=x]}.$$
With a slight abuse of notation, we define the min-entropy of a function $X \colon A \to B$, with $A$ and $B$ finite, in the same way (using the uniform measure):
$$\ME(X) = \min_{b \in B} \, \log_2\frac{1}{\Pr_{a \in A}[f(a)=b]}
= -\log_2 \max_{b \in B}\frac{|X^{-1}(b)|}{|A|}.$$
Note that we have the trivial bounds $0 \le \ME(X) \le \log_2|B|$, where the lower bounds holds when $X$ is constant and the upper bound when $X$ is $|A|/|B|$-to-$1$.

Recall that $\SR(f)$ denotes the slice rank of the $3$-tensor corresponding to $f$
(which can be thought of as the ``oracle complexity'' of $f$, where the oracle produces any desired, arbitrarily hard matrices).
Towards the proof of Proposition~\ref{main:complexity-bias}, 
we first deduce from Theorem~\ref{main:summary} a tight relation between slice rank and min-entropy for the class of bilinear maps.
\begin{proposition}\label{theo:SR-trade-off}
For any bilinear map $f \colon \FF^{n} \times \FF^n \to \FF^n$ over any 
finite field $\FF \neq \FF_2$,
$$\SR(f) = \Theta\Big(\frac{\ME(f)}{\log_2|\FF|}\Big) .$$
\end{proposition}
\begin{proof}
As $f$ is bilinear, we claim that
$\max_{\B{b}} \Pr_{\B{a}}[f(\B{a})=\B{b}] = \Pr_{\B{a}}[f(\B{a})=\B{0}]$.
Indeed, this follows from the fact that $f(\B{x},\B{y})$ is a linear map for any fixed $\B{y}$, and thus for every $\B{b}$,
\begin{align*}
	\Pr_{\B{x},\B{y}}[f(\B{x},\B{y})=\B{b}] 
	&= \Ex_{\B{y}} \Pr_{\B{x}}[f(\B{x},\B{y})=\B{b}]\\
	&\le \Ex_{\B{y}} \Pr_{\B{x}}[f(\B{x},\B{y})=\B{0}]
	= \Pr_{\B{x},\B{y}}[f(\B{x},\B{y})=\B{0}].
\end{align*}
Therefore, 
$$\ME(f) = \min_{\B{b}} -\log_2 \Pr_{\B{x},\B{y}}[f(\B{x},\B{y})=\B{b}]
=  -\log_2 \Pr_{\B{x},\B{y}}[f(\B{x},\B{y})=\B{0}]\\
= \AR(f)\log_2|\FF|$$
where the last equality uses Fact~\ref{fact:AR-ch}.
We deduce using Theorem~\ref{main:summary} that
$$\SR(f) = \Theta(\AR(f )) = \Theta(\ME(f)/\log_2|\FF|),$$
as desired.
\end{proof}

\begin{proof}[Proof of Proposition~\ref{main:complexity-bias}]
Note that, almost directly from the definitions, $\complexity^*(f) \le n\SR(T)$. The desired bound therefore follows from Proposition~\ref{theo:SR-trade-off},
$$\complexity^*(f) \le n\SR(T) = O(n\ME(f)/\log_2|\FF|).$$
Below we show that our bound is in fact an equality (up to a constant) for almost every bilinear map. 
Let $f \colon \FF^n \times \FF^n \to \FF^n$ be a uniformly random bilinear map. 
We have $\complexity^*(f) = \Theta(n^2)$, since the tensor rank of the corresponding tensor is $\Theta(n^2)$, which is equal to $\Theta(\complexity^*(f))$ for any $\FF$ large enough, as shown by Strassen~\cite{Strassen73} (see also~\cite{Blaser}).
Thus, we next show that $\ME(f)=\Theta(n\log_2|\FF|)$. 
Observe that if $L \colon \FF^n \to \FF^n$ is a uniformly random linear map then for any $\B{0} \neq \B{y} \in \FF^n$ we have that $L(\B{y})$ is uniformly random in $\FF^n$.
Fix $\B{0} \neq \B{y_0} \in \FF^n$. Then for each component $f_i(\B{x},\B{y})=:\B{x}^T A_i \B{y}$ of $f$ we have that $f_i(\B{x},\B{y_0}) = \B{x}^T (A_i \B{y_0})$ is a uniformly random linear form in $\B{x}$. Moreover, these $n$ linear forms $f_1(\B{x},\B{y_0}),\ldots,f_n(\B{x},\B{y_0})$ are independent.
It follows that $f(\B{x},\B{y_0}) \colon \FF^n \to \FF^n$ is a uniformly random linear map. Thus, $f(\B{x},\B{y_0})$ is a bijection.
We conclude that $|f^{-1}(\B{0})|= \big(\sum_{\B{0} \neq \B{y} \in \FF^n} 1\big) + |\FF|^n = 2|\FF|^n-1$ (and $|f^{-1}(\B{b})|=|\FF|^n-1$ for any $\B{b}\neq 0$).	
Therefore, $\ME(f) = \log_2 (|\FF|^{2n}/|f^{-1}(\B{0})|) = \Theta(\log_2(|\FF|^n)) = \Theta(n\log_2|\FF|)$,
as desired.
\end{proof}

\subsection{Approximating bilinear maps}\label{subsec:Approx}

Recall that maps $f,g \colon A \to B$ are said to be $\d$-close if $\Pr_{a \in A}[f(a) = g(a)] = \d$.
Let us recall the classical fact that, for any $\class{NP}$-complete function $f\colon\{0,1\}^* \to \{0,1\}$, say $f=\class{SAT}$, if $f$ can be computed in polynomial time on all but polynomially many inputs then in fact $f$ can be computed in polynomial time. Phrased differently, if $g\colon\{0,1\}^* \to \{0,1\}$ is such that $g_n:=g|_{\{0,1\}^n}$ is $\d$-close to $f_n:=f|_{\{0,1\}^n}$ with $\d=1-\poly(n)/2^n$ then $g \in \class{P}$ implies $f \in \class{P}$.
What would be an optimal analogue of this basic fact when $f$ is coming from the class of bilinear maps?
We note that this restriction is already a radical change of regime. For example, the Schwartz-Zippel lemma implies that if two distinct degree-$d$ forms are $\d$-close then necessarily $\d \le d/|\FF|$. In particular, an agreement that is close to $1$, as in the example above, is impossible in the bilinear setting.

Here we prove Proposition~\ref{main:reduction}, showing 
that it suffices to compute $f$ on a surprisingly small fraction of the inputs in order to be able to compute $f$ on all inputs.
For example, this implies that if $\SR(g) = O(r)$ and $g$ agrees with $f$ on merely an $|\FF|^{-O(r)}$-fraction of the inputs, then already $\SR(f) = O(r)$.
As before, Theorem~\ref{main:summary} supplies the precise bounds we need.

\begin{proof}[Proof of Proposition~\ref{main:reduction}]
Since $\SR$ is subadditive by definition, we have
$$\SR(f) = \SR(g+f-g) \le \SR(g)+\SR(f-g).$$
By Theorem~\ref{main:summary} we have 
$$\SR(f-g) \le O(\AR(f-g)) .$$
Write $\AR(f-g) = -\log_{|\FF|}\bias(f-g)$ and $\bias(f-g) =  \Pr_{\B{x},\B{y}}[(f-g)(\B{x},\B{y}) = \B{0}] = \d$. 
Combining the above inequalities gives  
$$\SR(f) - \SR(g) \le \SR(f-g) \le O(\AR(f-g)) = O(\log_{|\FF|} (1/\d)).$$
By symmetry, the same bound holds when interchanging $f$ and $g$, which proves the desired bound.

Finally, it remains to see that our bound is sharp,
for any value of $\SR(f),\SR(g)$.
Let $r,t$ be positive integers satisfying $t=\Theta(r)$, and let $n \ge r+t$. 
Let $f,g \colon \FF^n \times \FF^n \to \FF^n$ be the bilinear maps $$f(\B{x},\B{y})=(x_1y_1,\ldots,x_ry_r,0,\ldots,0)$$
and
$$g(\B{x},\B{y})=(0,\ldots,0,x_{r+1}y_{r+1},\ldots,x_{r+t}y_{r+t},0,\ldots,0).$$
Recall that for an identity tensor $I_m$ we have $\SR(I_m)=m$ (see, e.g.,~\cite{SawinTao16}) and $\AR(I_m) = \Theta(m)$. 
On the one hand, $|\SR(f)-\SR(g)|=|\SR(I_r)-\SR(I_t)|=|r-t|=\Theta(r)$. 
On the other hand, 
\begin{align*}
	\d &= \Pr_{\B{x},\B{y}}[f(\B{x},\B{y})=g(\B{x},\B{y})]\\ 
	&= \bias(f-g) = \bias(f)\cdot\bias(g) = |\FF|^{\Theta(r)+\Theta(t)}.
\end{align*}
Therefore, $\log_{|\FF|}(1/\d)=\Theta(r+t)=\Theta(r)$ as well, completing the proof.
\end{proof}

\begin{corollary}
Let $\FF \neq \FF_2$ be a finite field.
Any two bilinear maps $f,g\colon \FF^n\to\FF^m$ that are $\d$-close satisfy
$$\complexity^*(f) \le O((\SR(g) + \log_{|\FF|}(1/\d))n) .$$
\end{corollary}

\section{Discussion and open questions}\label{sec:open}

Several problems are left open by the results in this paper.
Of course, it would be interesting to extend our methods to higher-order tensors.
It would also be interesting to see other instantiations of classical results of theoretical computer science in the settings of bilinear, or more generally, low-degree polynomial maps.
It would be satisfying to extend our main result,  Theorem~\ref{main:summary}, to $\FF_2$. As of now, the best bound over $\FF_2$ remains $\SR(T) \le O(\AR(T)^4)$, and we wonder whether it might be that a linear upper bound simply does not hold $\FF_2$.

Finally, it remains open to determine the best possible constant $C$ such that $\SR(T) \le C\cdot\GR(T)$.
Let us show below that $C \ge 3/2$.
Over any field $\FF$, let $T \in \FF^{3 \times 3 \times 3}$ denote the Levi-Civita tensor $T=(\varepsilon_{i,j,k})_{i,j,k}$. In other words, the trilinear form corresponding to $T$ is the $3$-by-$3$ determinant polynomial,
$$T(\B{x},\B{y},\B{z})
=\det
\begin{pmatrix}
	x_1 & x_2 & x_3 \\
	y_1 & y_2 & y_3 \\
	z_1 & z_2 & z_3
\end{pmatrix}.$$
We will show that $\GR(T)=2$ and $\SR(T)=3$, giving the bound $C \ge \SR(T)/\GR(T) = 3/2$. 
To compute $\GR(T)$, observe that the bilinear map $f \colon \FF^3 \times \FF^3 \to \FF^3$ corresponding to $T$ is $f(\B{x},\B{y}) = \B{x} \times \B{y}$, that is, the cross product of the vectors $\B{x},\B{y} \in \FF^3$.
Therefore, $(\B{x},\B{y}) \in \ker f$ if and only if $\B{x} \times \B{y} = \B{0}$, that is, $\B{x}$ and $\B{y}$ are linearly dependent.
We deduce that $\GR(T)=\codim\ker f=2$, or equivalently $\dim\ker f = 4$, since $\B{y} \in \FF^3$ is completely determined by $\B{x} \in \FF^3$ together with a scalar multiple in $\FF$.
To compute $\SR(T)$, observe that $x_iy_jz_k$ is a monomial of $T(\B{x},\B{y},\B{z})$ if and only if $i,j,k \in [3]$ are all distinct.
Let $S = \{(i,j,k) \in [3]^3 \mid x_iy_jz_k \text { is in the support of } T\}$.
Observe that $S$ forms an antichain; indeed, $i+j+k=6$ is constant for all $(i,j,k) \in S$.
Thus, by Proposition~4 in~\cite{SawinTao16}, $\SR(T)$ is equal to the vertex cover number of $S$ when viewed as a ($3$-partite) $3$-uniform hypergraph.
Since each vertex of the hypergraph $S$ has degree exactly $2$, it follows that any vertex cover has at least $3!/2$ vertices. We deduce that $\SR(T)=3$.

One can actually obtain an infinite family of $3$-tensors with a similar ratio, implying that $\SR(T)/\GR(T)$ does not drop below $3/2$ even for large tensors. For any $k \in \NN$, let $T_k \in \FF^{3k \times 3k \times 3k}$ be the $k$-fold direct sum of $T$ with itself. We have $\GR(T_k)=k\cdot \GR(T)=2k$ by the additivity of $\GR$ with respect to direct sums (see Lemma~4.3 in~\cite{KoppartyMoZu20}). Moreover, we have $\SR(T_k)=k\cdot\SR(T)=3k$ since the hypergraph corresponding to the support of $T_k$ is a disjoint union of copies of the hypergraph corresponding to the support of $T$, and thus is also a $2$-regular antichain. Therefore, any vertex cover has at least $6k/2$ vertices, implying that $\SR(T_k)=3k$.

Let us end by noting a curious analogy between $\GR/\SR$ and two other notions of rank for $3$-tensors, \emph{commutative rank}/\emph{non-commutative rank}. 
It is known that non-commutative rank is at most twice the commutative rank, which interestingly matches the constant $2$ in our Theorem~\ref{theo:main}.
Moreover, just like in this paper, constructions were given that witness a $3/2$ lower bound, and it was conjectured that $3/2$ might be the correct constant~\cite{FortinRe04}. However, this was recently refuted by Derksen and Makam~\cite{DerksenMa18}, whose construction achieves a ratio that is arbitrarily close to $2$.  
It would be interesting to understand whether there is a deeper analogy between these two pairs of ranks!

\paragraph*{Note added in proof.} 
Since the submission (November 6, 2020) to the 53rd ACM Symposium on Theory of Computing, bounds for higher-order tensors over large enough fields were 
obtained by the authors, using considerably more intricate arguments~\cite{CohenMo21}.
Additionally, similar results to the ones in the current paper for $3$-tensors were obtained by Adiprasito, Kazhdan, and Ziegler~(\cite{AdiprasitoKaZi21}, February 6, 2021).\footnote{Their proof relies in part on a generalization of an argument found in a number-theoretic paper of Schmidt~\cite{Schmidt85}; in the case of $3$-tensors, Schmidt's argument gives an inequality analogous to $\SR(T) \le O(\GR(T))$ over the complex numbers.}

%%% AUTHOR: optional appendix here
%\appendix %% you may comment this out if no Appendix
%\section*{Appendix}
%\section{Improving the constants}
%Material is placed here as needed.

%%% AUTHOR: optional acknowledgments here
\section*{Acknowledgments} %%  you may comment this out if no Ackno
We thank the anonymous reviewers for a careful reading and useful suggestions.

%%% AUTHOR:
%%% Bibliography goes here. Note that the arXiv cannot process bibtex
%%% or biber bibliographies.  Example of acceptable bibliograpy format:
\bibliographystyle{amsplain}

%% AUTHOR: You can generate such a bibliography from a .bib file by 
%% running pdflatex/bibtex/pdflatex/pdflatex and then pasting the .bbl file
%% between \begin{thebibliography} and \end{bibliography}

%%% AUTHOR: Include a short description of each author following the
%%% structure below. Use the same short tags used previously.  
%%% Use \imageat{} and \imagedot{} instead of "@" and "." in
%%% email addresses-this replaces the symbols with graphics to avoid 
%%% e-mail address harvesting from the .pdf file
\begin{dajauthors}
\begin{authorinfo}[alex]
  Alex Cohen\\
  %Graduate student\\
  Massachusetts Institute of Technology\\
  Cambridge, MA, USA\\
  alexcoh\imageat{}mit\imagedot{}edu\\
  \url{https://math.mit.edu/~alexcoh}
\end{authorinfo}
\begin{authorinfo}[guy]
	Guy Moshkovitz\\
	%Professor\\
	Department of Mathematics\\
	City University of New York (Baruch College)\\
	New York, NY, USA\\
	guymoshkov\imageat{}gmail\imagedot{}com\\
	\url{https://sites.google.com/view/guy-moshkovitz}
\end{authorinfo}
\end{dajauthors}

\end{document}